\documentclass[11pt,a4paper]{amsart}
\usepackage{amssymb,amsmath,amsfonts}

\headsep=1.2500cm \numberwithin{equation}{section}
\newtheorem{Theorem}{Theorem}[section]
\newtheorem{Proposition}[Theorem]{Proposition}
\newtheorem{cor}[Theorem]{Corollary}

\theoremstyle{remark}
\newtheorem{Definition}[Theorem]{Definition}
\newtheorem{Example}[Theorem]{Example}
\newtheorem{Remark}[Theorem]{Remark}

\begin{document}

\title{On (approximate) homological  notions of certain Banach algebras}

  \author[A. Sahami]{A. Sahami}
  \email{a.sahami@ilam.ac.ir}
  
  \address{Department of Mathematics
  	Faculty of Basic Sciences Ilam University P.O. Box 69315-516 Ilam,
  	Iran.}

\keywords{Approximate $\phi-$contratiblity, Approximate biprojectivity, $\phi-$biflatness, $\phi-$biprojectivity, Banach algebra.}

\subjclass[2010]{ 46M10, 46H20,  46H05.}

\maketitle
%-------------------------------------------------------------

%%%%%%%%%%%%%%%%%%%%%%%%%%%%%%%%%%%%%%%%%%%%%%%%%%%%%%%%%%%%%%%%%%%%%%%%%
\begin{abstract}
In this paper, we study the notion of  $\phi$-biflatness, $\phi$-biprojectivity,  approximate biprojectivity and Johnson pseudo-contractibility for a new class of Banach algebras. Using this class of Banach algebras we give some examples which are  approximately biprojective. Also some Banach algebras are given among matrix algebras which are never Johnson pseudo-contractible.
\end{abstract}
\section{Introduction and preliminaries}
Given a Banach algebra $A$, Kamyabi-Gol {\it et
	al.} in \cite{kam} defined a  new product on $A$ which is denoted by $\ast$. In fact $a\ast b=aeb,$ for each $a,b\in A,$ where $e$ is an element of the closed unit ball $\overline{B^{0}_{1}}$ of $A$. A Banach algebra $A$ equipped with $\ast$ as its product is denoted by $A_{e}$. They studied some properties like amenability and Arens regularity  of $A_{e}$. In \cite{khod} some homological properties of $A_{e}$ like biflatness, biprojectivity and $\phi-$amenability discussed.

New notions of $\phi-$amenability and approximate notions of homological Banach theory introduced and studied for Banach algebras see\cite{sam}, \cite{zhang} and \cite{kan}. In fact a Banach algebra 
a Banach algebra $A$ is called approximate $\phi-$contractible if there exists a net ($m_{\alpha})$ in $A$ such that $am_{\alpha}-\phi(a)m_{\alpha}\rightarrow 0,$ and $\phi(m_{\alpha})=1,$ for every $a\in A,$ where $\phi$ is a multiplicative linear functional on $A$. For more information see \cite{agha}.  Also a Banach algebra $A$ is called   
	approximate biprojective if there exists a net of
bounded linear maps from $A$ into $A\otimes_{p}A$, say
$(\rho_{\alpha})_{\alpha\in I}$, such that
\begin{enumerate}
	\item [(i)] $a\cdot \rho_{\alpha}(b)-\rho_{\alpha}(ab)\xrightarrow{||\cdot||} 0$,
	\item [(ii)] $\rho_{\alpha}(ba)-\rho_{\alpha}(b)\cdot a\xrightarrow{||\cdot||} 0$,
	\item [(iii)] $\pi_{A}\circ\rho_{\alpha}(a)-a\rightarrow
	0$,
\end{enumerate}
for every $a,b\in A$. In \cite{agha-aust} the structure of approximate biprojective Banach algebras and its nilpotent ideals and also the relation with other notions of amenability are discussed.  

We present some standard notations and definitions that we shall need
in this paper. Let $A$ be a Banach algebra. Throughout this work,
the character space of $A$ is denoted by $\Delta(A)$, that is, all
non-zero multiplicative linear functionals on $A$.
 For each
$\phi\in\Delta(A)$ there exists a unique extension $\tilde{\phi}$ to
$A^{**}$ which is defined $\tilde{\phi}(F)=F(\phi)$. It is easy to
see that $\tilde{\phi}\in\Delta(A^{**})$.
 The projective
tensor product
$A\otimes_{p}A$ is a Banach $A$-bimodule via the following actions
$$a\cdot(b\otimes c)=ab\otimes c,~~~(b\otimes c)\cdot a=b\otimes
ca\hspace{.5cm}(a, b, c\in A).$$ The product morphism $\pi_{A}:A\otimes_{p}A\rightarrow A$ is given by $\pi_{A}(a\otimes b)=ab,$ for every $a,b\in A.$ 

Let $A$ and $B$ be Banach algebras.
We denote by
$\phi\otimes \psi$  a  map defined by $\phi\otimes \psi(a\otimes
b)=\phi(a)\psi(b)$ for all $a\in A$ and $b\in B.$ It is easy to see
that $\phi\otimes \psi\in\Delta(A\otimes_{p}B).$

 Let $X$ and $Y$ be Banach $A-$bimodules. The map $T:X\rightarrow Y$ is called $A-$bimodule morphism, if 
$$T(a\cdot x)=a\cdot T(x),\quad T(x\cdot a)=T(x)\cdot a,\qquad (a\in A,x\in X).$$
Also a net of $(T_{\alpha})$ of maps from $X$ into $Y$ is called approximate $A-$bimodule morphism, if 
$$T_{\alpha}(a\cdot x)-a\cdot T_{\alpha}(x)\rightarrow 0,\quad T_{\alpha}(x\cdot a)-T_{\alpha}(x)\cdot a\rightarrow 0,\qquad (a\in A,x\in X).$$

The content of the paper is as follows. In section 2 we study $\phi-$homological properties of $A_{e}$ like $\phi-$biflatness and $\phi-$biprojectivity.  Approximate biprojectivity and Johnson pseudo-contractibility are two important notions of Banach homology theory, which we discuss for $A_{e}$ in section 3. We give some examples of matrix algebras to illustrate  the paper.
%------------------------------------------------------------------------------------------------------------------------------------------
%%%%%%%%%%%%%%%%%%%%%%%%%%%%%%%%%%%%%%%%%%%%%%%%%%%%%%%%%%%%%%%%%%%%%%%%%%%%%%%%%%%%%%%%%%%%%%%%%%%%%%%%%%%%%%%%%%%%%%%%%%%%%%%%%%%%%%%%%%%
%------------------------------------------------------------------------------------------------------------------------------------------
\section{$\phi-$homological properties of certain Banach algebras}
This section is devoted to the concepts of Banach homology related to a charater $\phi.$
\begin{Proposition}\cite[Proposition 2.3]{kam}\label{exis unit}
Let A be a Banach algebra and $e\in \overline{B^{0}_{1}}$ . Then $A_{e}$ is unital if and only if A is unital and $e$ is invertible.
\end{Proposition}
\begin{Proposition}\cite[Proposition 2.4]{kam}
	Let A be a Banach algebra and $e\in \overline{B^{0}_{1}}$ .Then the followings hold:
	\begin{enumerate}
		\item If $\phi$ is a multiplicative linear functional on A, then  $\phi(e)\phi$ is a
		multiplicative linear functional on $A_{e}$.
		\item If $A_{e}$ is unital and  $\psi$ is a multiplicative linear functional on $A_{e}$, then
		$\phi(a)=\psi(e^{-1}a)$ is a multiplicative linear functional on A.
	\end{enumerate}
\end{Proposition}
\begin{Proposition}\cite[Proposition 2.3]{khod}\label{khod}
	Let A be a Banach algebra and  $e\in \overline{B^{0}_{1}}$. If $A_{e}$ is unital
	then $(A_{e})_{e^{-2}}$ = A, ( isometrically isomorphism ).
\end{Proposition}
\begin{Proposition}\label{app left}
Suppose that $A$ is a Banach algebra and also suppose that $e\in \overline{B^{0}_{1}}$ and $\phi\in\Delta(A)$. Then the followings hold:
\begin{enumerate}
\item If $A$ is approximate $\phi-$contractible and $\phi(e)\neq 0$, then $A_{e} $ is approximately  $\psi$-contractible, where $\psi=\phi(e)\phi$.
\item If $A_{e}$ is unital and approximate $\psi-$contractible, then $A$ is approximate $\phi$-contractible, where $\phi(a)=\psi(e^{-1}a)$ for each $a\in A$.
\end{enumerate}
\end{Proposition}
\begin{proof}
Suppose that $A$ is approximately $\phi-$contractible. So there is a net $(m_{\alpha})$ in $A$ such that $$am_{\alpha}-\phi(a)m_{\alpha}\rightarrow 0,\quad \phi(m_{\alpha})=1,\qquad (a\in A).$$ Define $n_{\alpha}=\frac{m_{\alpha}}{\phi(e)}$. Thus we have 
\begin{equation*}
\begin{split}
a\ast n_{\alpha}-\psi(a)n_{\alpha}&=aen_{\alpha}-\psi(a)n_{\alpha}\\
&=ae\frac{m_{\alpha}}{\phi(e)}-\psi(a)\frac{m_{\alpha}}{\phi(e)}\\
&=ae\frac{m_{\alpha}}{\phi(e)}-\phi(ae)\frac{m_{\alpha}}{\phi(e)}+\phi(ae)\frac{m_{\alpha}}{\phi(e)}-\psi(a)\frac{m_{\alpha}}{\phi(e)}\rightarrow 0,\qquad (a\in A_{e}).
\end{split}
\end{equation*}
Also $$\psi(n_{\alpha})=\psi(\frac{m_{\alpha}}{\phi(e)})=\phi(e)\phi(\frac{m_{\alpha}}{\phi(e)})=\phi(m_{\alpha})=1.$$ It follows that $A_{e}$ is approximate  $\phi-$contractible. 

Suppose that $\phi(a)=\psi(e^{-1}a)$ and also suppose that $A_{e}$ is unital and approximately left $\psi-$contractible. It is easy to see that $\psi(a)=\phi(ea)$.  Let $(m_{\alpha})$ be a net in $A_{e} $ such that 
$$a\ast m_{\alpha}-\psi(a)m_{\alpha}\rightarrow 0,\quad \psi(m_{\alpha})=1,\qquad (a\in A_{e}).$$ Since
\begin{equation*}
\begin{split}
a\ast m_{\alpha}-\psi(a)m_{\alpha}&=aem_{\alpha}-\psi(a)m_{\alpha}\\
&=aem_{\alpha}-\phi(ea)m_{\alpha}\\
&=aem_{\alpha}-\phi(e)\phi(a)m_{\alpha}\\
&=aem_{\alpha}-\phi(a)\phi(e)m_{\alpha}\\
&=aem_{\alpha}-\phi(ae)m_{\alpha},
\end{split}
\end{equation*}
we have 
$$a\ast m_{\alpha}-\psi(a)m_{\alpha}=aem_{\alpha}-\phi(ae)m_{\alpha}\rightarrow 0$$ for each $a\in A.$ Replacing $a$ with $ae^{-1}$ we have $am_{\alpha}-\phi(a)m_{\alpha}\rightarrow 0.$ Regarding  $$1=\psi(m_{\alpha})=\phi(em_{\alpha})=\phi(e)\phi(m_{\alpha}),$$
we may suppose that $\phi(m_{\alpha})\neq 0,$ for each $\alpha.$ Now define $n_{\alpha}=\frac{m_{\alpha}}{\phi(m_{\alpha})}$. Clearly $\phi(n_{\alpha})=1.$ Also $$an_{\alpha}-\phi(a)n_{\alpha}=a\frac{m_{\alpha}}{\phi(m_{\alpha})}-\phi(a)\frac{m_{\alpha}}{\phi(m_{\alpha})}\rightarrow 0.$$ It finishes the proof.
\end{proof}
\begin{Example}
	Let  $A=\{\left(\begin{array}{ccc} a_{11}&a_{12}&a_{13}\\
	0&a_{22}&a_{23}\\
	0&0&a_{33}
	\end{array}
	\right)| a_{ij}\in \mathbb{C}\}$ and suppose that $e=\left(\begin{array}{ccc} \frac{1}{6}&\frac{1}{6}&\frac{1}{6}\\
	0&\frac{1}{6}&\frac{1}{6}\\
	0&0&\frac{1}{6}
	\end{array}
	\right)$. Clearly $A$ with matrix operations and $\ell^1$-norm is a Banach algebra. We know that  $e$ is invertible and by Proposition \ref{exis unit}, $A_{e}$ is unital. Define $\phi:A\rightarrow \mathbb{C}$ by $$\phi(\left(\begin{array}{ccc} a_{11}&a_{12}&a_{13}\\
	0&a_{22}&a_{23}\\
	0&0&a_{33}
	\end{array}
	\right))=a_{33}.$$ Clearly $\phi$ is a character(multiplicative linear functional) and $\phi(e)\neq 0$. Suppose conversely that $A_e$ is approximate $\psi-$contractible. By previous Proposition(2), $A$ becomes approximate $\phi-$contractible. On the other hand by the same arguments as in  the proof of \cite[Theorem 5.1]{sah taki}
 $A$ is not approximate $\phi-$contractible, which is a contradiction.
\end{Example}
Let $A$ be a Banach algebra and $\phi\in \Delta(A)$. $A$ is called
$\phi$-biprojective, if there exists a bounded $A$-bimodule morphism
$\rho:A\rightarrow A\otimes_{p}A$ such that $\phi \circ\pi_{A}
\circ\rho=\phi$. Also $A$ is called $\phi$-biflat if there exists a
bounded $A$-bimodule morphism $\rho:A\rightarrow (A\otimes_{p}
A)^{**}$ such that $\tilde{\phi} \circ \pi_{A}^{**} \circ\rho=\phi$. For more information about $\phi-$biflatness and $\phi-$biprojectivity, the reader refers to \cite{sah phi-biflat} and \cite{sah col}.

\begin{Theorem}
	Let $A$ be a Banach algebra and $\phi\in\Delta(A)$. Suppose that $e\in \overline{B^{0}_{1}}$ and $\phi(e)\neq 0.$
	If $A$ is $\phi$-biprojective, then $A_{e}$ is $\psi=\phi(e)\phi$-biprojective.
\end{Theorem}
\begin{proof}
Since 	$A$ is $\phi$-biprojective, there exists  a bounded $A-$bimodule morphism $\rho:A\rightarrow A\otimes_{p}A$ such that $\phi\circ\pi_{A}\circ\rho =\phi.$ Define $\tilde{\rho}=\frac{1}{\phi(e)}\rho$. We show that $\tilde{\rho}$ is a bounded $A_e$-bimodule morphism. To see this, consider
\begin{equation*}
\begin{split}
\tilde{\rho}(a\ast b)=\frac{1}{\phi(e)}\rho(a\ast b)=\frac{1}{\phi(e)}\rho(a e b)&=ae\frac{1}{\phi(e)}\rho( b)\\
&=a\ast \frac{1}{\phi(e)}\rho( b)\\
&=a\ast \tilde{\rho}( b),\qquad (a,b\in A_{e}).
\end{split}
\end{equation*}
Also
\begin{equation*}
\begin{split}
\tilde{\rho}(a\ast b)=\frac{1}{\phi(e)}\rho(a\ast b)=\frac{1}{\phi(e)}\rho(a e b)&=\frac{1}{\phi(e)}\rho( a)be\\
&= \frac{1}{\phi(e)}\rho( a)\ast b\\
&= \tilde{\rho}( a)\ast b,\qquad (a,b\in A_{e}).
\end{split}
\end{equation*}
On the other hand, since $$\psi\circ \pi_{A_{e}}\circ \tilde{\rho}=\phi(e)\phi\circ\pi_{A}\circ\rho,$$
we have 
\begin{equation*}
\begin{split}
\psi\circ \pi_{A_{e}}\circ\tilde{\rho}(a)=\phi(e)\phi\circ\pi_{A}\circ\rho(a)=\phi(e)\phi(a)=\psi(a),\quad (a\in A_{e}).
\end{split}
\end{equation*}
So $A_{e}$ is $\psi-$biprojective.
	\end{proof}
Using the similar arguments as in the proof of the previous theorem, we have the following corollary:
\begin{cor}
	Let $A$ be a Banach algebra and $\phi\in\Delta(A)$. Suppose that $e\in \overline{B^{0}_{1}}$ and $\phi(e)\neq 0.$
	If $A$ is $\phi$-biflat, then $A_{e}$ is $\psi=\phi(e)\phi$-biflat.
\end{cor}
Let $A$ be a Banach algebra and $\phi\in \Delta(A)$.  $A$ is called $\phi-$amenable if there exists a bounded net $(m_{\alpha})$ in $A$ such that $am_{\alpha}-\phi(a)m_{\alpha}\rightarrow 0$ and $\phi(m_{\alpha})=1,$ for every $a\in A,$ see \cite{kan}.
\begin{cor}
	Let $A$ be a Banach algebra and $\phi\in\Delta(A)$. Suppose that $e\in \overline{B^{0}_{1}}$ and $\phi(e)\neq 0.$
	If $A$ is $\phi$-biflat and $A$ has a left approximate identity, then $A_{e}$ is approximate  $\psi=\phi(e)\phi$-contractible.
\end{cor}
\begin{proof}
Since  $A$ is $\phi$-biflat and $A$ has a left approximate identity, by similar arguments as in the proof of \cite[Theorem 2.2]{sah taki} $A$ is  $\phi$-amenable. It is easy to see that   $\phi$-amenability of $A$ implies that $A$ is approximate $\phi-$contractible. Applying Proposition \ref{app left}, $A_{e}$ becomes approximate $\psi-$contractible.
\end{proof}
Let $A$ b a Banach algebra and $\phi\in\Delta(A).$ Then
 $A$ is called   
approximate left $\phi$-biprojective if there exists a net of
bounded linear maps from $A$ into $A\otimes_{p}A$, say
$(\rho_{\alpha})_{\alpha\in I}$, such that
\begin{enumerate}
	\item [(i)] $\rho_{\alpha}(ab)-\phi(a)\rho_{\alpha}(b)\xrightarrow{||\cdot||} 0$,
	\item [(ii)] $\rho_{\alpha}(ba)-\rho_{\alpha}(b)\cdot a\xrightarrow{||\cdot||} 0$,
	\item [(iii)] $\pi_{A}\circ\rho_{\alpha}(a)-a\rightarrow
	0$,
\end{enumerate}
for every $a,b\in A$, see \cite{sah glas}.
\begin{Theorem}\label{app phi proj}
	Let $A$ be a Banach algebra and $\phi\in\Delta(A)$. Suppose that $e\in \overline{B^{0}_{1}}$ and $\phi(e)\neq 0.$
	If $A$ is approximate left $\phi$-biprojective, then $A_{e}$ is approximate left $\psi=\phi(e)\phi$-biprojective.
\end{Theorem}
\begin{proof}
	Since 	$A$ is approximate left $\phi$-biprojective, there exists  a net of bounded linear maps $(\rho_{\alpha})$ from $A$ into $A\otimes_{p}A$ such that
	$$\rho_{\alpha}(ab)-\phi(a)\rho_{\alpha}(b)\rightarrow 0,\quad \rho_{\alpha}(ab)-\rho_{\alpha}(a)\cdot b\rightarrow 0,\quad \phi\circ\pi_{A}\circ\rho(a)-\phi(a)\rightarrow 0.$$ Define $\tilde{\rho}_{\alpha}=\frac{1}{\phi(e)}\rho_{\alpha}$. We show that there exists  a net of bounded linear maps $(\tilde{\rho}_{\alpha})$ from $A_{e}$ in to $A_{e}\otimes_{p}A_{e}$ such that
	$$\tilde{\rho}_{\alpha}(a\ast b)-\psi(a)\tilde{\rho}_{\alpha}(b)\rightarrow 0,\quad \tilde{\rho}_{\alpha}(a\ast b)-\tilde{\rho}_{\alpha}(a)\ast b\rightarrow 0,\quad \psi\circ\pi_{A}\circ\tilde{\rho}(a)-\psi(a)\rightarrow 0.$$To see this, consider
	\begin{equation*}
	\begin{split}
	\tilde{\rho}_{\alpha}(a\ast b)-\psi(a)\tilde{\rho}_{\alpha}(b)
	&=\tilde{\rho}_{\alpha}(ae b)-\phi(a)\phi(e)\tilde{\rho}_{\alpha}(b)\\
	&=\frac{1}{\phi(e)}(\rho_{\alpha}(ae b)-\phi(a)\phi(e)\rho_{\alpha}(b))\\
	&=\frac{1}{\phi(e)}(\rho_{\alpha}(ae b)-\phi(ae)\rho_{\alpha}( b)+\phi(ae)\rho_{\alpha}( b)-\phi(a)\phi(e)\rho_{\alpha}(b))\\
	&\rightarrow 0
	\end{split}
	\end{equation*}
	Also
	\begin{equation*}
	\begin{split}
	\tilde{\rho}_{\alpha}(a\ast b)-\tilde{\rho}_{\alpha}(a)\ast b=\frac{1}{\phi(e)}\rho_{\alpha}(ae b)-\frac{1}{\phi(e)}\rho_{\alpha}(a)eb\rightarrow 0
		\end{split}
	\end{equation*}
	On the other hand, since $$\psi\circ \pi_{A_{e}}\circ \tilde{\rho}_{\alpha}=\phi(e)\phi\circ\pi_{A}\circ\rho_{\alpha},$$
	we have 
	\begin{equation*}
	\begin{split}
	\psi\circ \pi_{A_{e}}\circ\tilde{\rho}_{\alpha}(a)-\psi(a)=\phi(e)\phi\circ\pi_{A}\circ\rho_{\alpha}(a)-\phi(e)\phi(a)\rightarrow \phi(e)\phi(a)- \phi(e)\phi(a)=0,\quad (a\in A_{e}).
	\end{split}
	\end{equation*}
	So $A_{e}$ is approximate left $\psi-$biprojective.
\end{proof}
\begin{Remark}\label{remark1}
	Let $A$ and $B $ be  Banach algebras and $e_{A}\in  \overline{B^{0}_{1}}^{A}$ and $e_{B}\in  \overline{B^{0}_{1}}^{B}$. Then there exist two sequences $(x_{n})$ and $(y_{n})$ in the unit ball $A$ and the unit ball B such that $x_{n}\rightarrow e_{A} $ and $y_{n}\rightarrow e_{B},$ respectively. Since $$||x_{n}\otimes y_{n}-e_{A}\otimes e_{B}||\leq ||x_{n}\otimes y_{n}-e_{A}\otimes y_{n}||+||e_{A}\otimes y_{n}-e_{A}\otimes e_{B}||\rightarrow 0,$$
	we have $e_{A}\otimes e_{B}\in  \overline{B^{0}_{1}}^{A\otimes_{p}B}$.  Define $T:A_{e_{A}}\otimes_{p}B_{e_{B}}\rightarrow {A\otimes_{p}B}_{e_{A}\otimes e_{B}}$ by 
	$T(a\otimes b)=a\otimes b$ for every $a\in A$ and $b\in B.$ It is easy to see that $T$ is an isometric algebra isomorphism. Also $T$ is a bounded ${A\otimes_{p}B}_{e_{A}\otimes e_{B}}-$bimodule morphism.
\end{Remark}
\begin{Proposition}
	Let $A$ and $B $ be  Banach algebras and $e_{A}\in  \overline{B^{0}_{1}}^{A}$ and $e_{B}\in  \overline{B^{0}_{1}}^{B}$. Suppose that $\phi_{A}\in\Delta(A)$ and $\phi_{B}\in \Delta(B)$ which $\phi_{A}(e_{A})\neq 0$ and $\phi_{B}(e_{B})\neq 0.$ If $A$ and $B$ are $\phi_{A}-$biprojective and $\phi_{B}-$biprojective, repectively, then 	${A\otimes_{p}B}_{e_{A}\otimes e_{B}}$ is $\phi_{A}(e_A)\phi_{A}\otimes \phi_{B}(e_B)\phi_{B}-$biprojective.
\end{Proposition}
\begin{proof}
Since $A$ and $B$ are $\phi_{A}-$biprojective and $\phi_{B}-$biprojective, repectively, then by Theorem \ref{app phi proj}, $A_{e}$ and $B_{e} $ are $\phi_{A}(e_A)\phi_{A}-$biprojective and $\phi_{B}(e_B)\phi_{B}-$biprojective, respectively. So there exist a $A_{e_{A}}-$bimodule morphism
$\rho_{0}:A_{e_{A}}\rightarrow A_{e_{A}}\otimes_{p}A_{e_{A}}$ and a $B_{e_{B}}$-bimodule morphism $\rho_{1}:B_{e_B}\rightarrow
B_{e_B}\otimes_{p}B_{e_B}$  such that $\phi_{A}(e_A)\phi_{A} \circ\pi_{A}\circ\rho_{0}=\phi_{A}(e_A)\phi_{A}$
and $\phi_{B}(e_B)\phi_{B} \circ\pi_{B} \circ\rho_{1}=\phi_{B}(e_B)\phi_{B}$. 
Define
$\theta:(A_{e_{A}}\otimes_{p}A_{e_{A}})\otimes_{p}(B_{e_{B}}\otimes_{p}B_{e_{B}})\rightarrow(A_{e_{A}}\otimes_{p}B_{e_{B}})\otimes_{p}(A_{e_{A}}\otimes_{p}B_{e_{B}})
$ by $$(a_{1}\otimes a_{2})\otimes (b_{1}\otimes b_{2})\mapsto
(a_{1}\otimes b_{1})\otimes (a_{2}\otimes b_{2}),$$ where $a_{1},
a_{2}\in A$ and $b_{1}, b_{2}\in B$. Clearly $\theta $ is an isometric algebra isomorphism.  Set $\rho=(T\otimes T)\circ \theta
\circ(\rho_{0}\otimes\rho_{1})\circ T^{-1}$,  where $T$ is the map defined as in Remark \ref{remark1}. We know that  $\rho$ is a bounded linear map from $ A\otimes_{p}B_{e_{A}\otimes e_{B}}$ into $ (A\otimes_{p}B_{e_{A}\otimes e_{B}})\otimes_{p}(A\otimes_{p}B_{e_{A}\otimes e_{B}})$. Consider
$$\pi_{A\otimes_{p}B_{e_{A}\otimes e_{B}}}\circ\theta(a_{1}\otimes a_{2}\otimes b_{1}\otimes
b_{2})=\pi_{A\otimes_{p}B_{e_{A}\otimes e_{B}}}(a_{1}\otimes b_{1}\otimes a_{2}\otimes
b_{2})=\pi_{A_{e_{A}}}(a_{1}\otimes a_{2})\otimes\pi_{B_{e_{B}}}(b_{1}\otimes b_{2}),$$
then clearly one can show that
$\pi_{A\otimes_{p}B_{e_{A}\otimes e_{B}}}\circ\theta=\pi_{A_{e_{A}}}\otimes\pi_{B_{e_{B}}}$. Hence,
$$\pi_{A\otimes_{p}B_{e_{A}\otimes e_{B}}}\circ\theta(\rho_{0}(a)\otimes\rho_{1}(b))=\pi_{A_{e_{A}}}\circ\rho_{0}(a)\otimes\pi_{B_{e_{B}}}\circ\rho_{1}(b)$$
and it is easy to see that $$\phi_{A}(e_A)\phi_{A}\otimes \phi_{B}(e_B)\phi_{B}
\circ\pi_{A\otimes_{p}B}\circ\theta(\rho_{0}\otimes\rho_{1})(a\otimes
b)=\phi_{A}(e_A)\phi_{A}\otimes \phi_{B}(e_B)\phi_{B}(a\otimes b),$$  the proof is complete.	
	\end{proof}
%------------------------------------------------------------------------------------------------------------------------------------------
%%%%%%%%%%%%%%%%%%%%%%%%%%%%%%%%%%%%%%%%%%%%%%%%%%%%%%%%%%%%%%%%%%%%%%%%%%%%%%%%%%%%%%%%%%%%%%%%%%%%%%%%%%%%%%%%%%%%%%%%%%%%%%%%%%%%%%%%%%%
%------------------------------------------------------------------------------------------------------------------------------------------
\section{Approximate homological properties of certain Banach algebras}
In this section we investigate approximate biprojectivity and Johnson pseudo-contractibility of $A_{e}$.
\begin{Theorem}\label{Approximate homological}
Suppose that $A$ is a Banach algebra and also suppose that $e\in \overline{B^{0}_{1}}$. Then the followings hold:
\begin{enumerate}
\item If $A$ is approximately biprojective and $A_{e}$ is unital then $A_{e}$ is approximately biprojective.
\item If $A_{e}$ is unital and approximately biprojective, then $A$ is approximately biprojective.
\end{enumerate}
\end{Theorem}
\begin{proof}
To show (1), suppose that $A$ is approximately biprojective and $A_{e}$ is unital. It follows that there is an approximately $A-$bimodule morphism $(\rho_{\alpha}) $ from $A$ into $A\otimes_{p}A$ such that $\pi_{A}\circ\rho_{\alpha}(a)-a\rightarrow 0$ for each $a\in A.$ Note that 
\begin{equation*}
\begin{split}
\rho_{\alpha}(a\ast b)-a\ast \rho_{\alpha}(b)&=\rho_{\alpha}(aeb)-a\ast \rho_{\alpha}(b)\\
&=\rho_{\alpha}(aeb)-ae\rho_{\alpha}(b)+ae\rho_{\alpha}(b)-a\ast \rho_{\alpha}(b)\rightarrow 0,
\end{split}
\end{equation*}
and 
\begin{equation*}
\begin{split}
\rho_{\alpha}(a\ast b)- \rho_{\alpha}(a)\ast b&=\rho_{\alpha}(aeb)-\rho_{\alpha}(a)\ast b\\
&=\rho_{\alpha}(aeb)-\rho_{\alpha}(a)eb+\rho_{\alpha}(a)eb- \rho_{\alpha}(a)\ast b\rightarrow 0,
\end{split}
\end{equation*}
for each $a\in A_{e}.$ It implies that $(\rho_{\alpha})$ from $A_{e}$ into $A_{e}\otimes_{p}A_{e}$ is an approximately $A_{e}$-bimodule morphism. Define $T:A_{e}\otimes_{p}A_{e}\rightarrow A_{e}\otimes_{p}A_{e}$ by $T(a\otimes b)=ae^{-1}\otimes b.$ Note that using Proposition \ref{exis unit},  the definition of $T$ makes sense. It is easy to see that $$T(a\ast (c\otimes d))=a\ast T(c\otimes d),\quad T((c\otimes d)\ast a)=T(c\otimes d)\ast a,\qquad (a,c,d\in A).$$ Set $\tilde{\rho}_{\alpha}=T\circ \rho_{\alpha}$. Using direct calculations we can see that  $$\pi_{A_{e}}\circ \tilde{\rho}_{\alpha}=\pi_{A}\circ\rho_{\alpha}$$
It follows that 
$$\pi_{A_{e}}\circ \tilde{\rho}_{\alpha}-a=\pi_{A}\circ\rho_{\alpha}-a\rightarrow 0,\qquad (a\in A_{e}).$$
Thus $A_{e}$ is approximately biprojective.

To show (2), suppose that $A_{e}$ is unital and approximately biprojective. By Proposition \ref{khod}, we know that $A=(A_{e})_{e^{-2}}$. Now applying (1) it is easy to see that $A$ is approximately biprojective.
\end{proof}
A Banach algebra $A$ is called biprojective if there exists a bounded $A-$bimodule morphism $\rho:A\rightarrow A\otimes_{p}A$ such that $\pi_{A}\circ\rho(a)=a$ for each $a\in A,$ see  \cite{run}.
\begin{Example}
	Let  $A=\{\left(\begin{array}{ccc} a_{11}&a_{12}&a_{13}\\
	a_{21}&a_{22}&a_{23}\\
	a_{31}&a_{32}&a_{33}
	\end{array}
	\right)| a_{ij}\in \mathbb{C}\}$. With the matrix operations  and $\ell^1$-norm, $A$ becomes a Banach algebra.  Suppose that $e=\left(\begin{array}{ccc} \frac{1}{4}&\frac{1}{4}&0\\
	0&\frac{1}{4}&0\\
	0&0&\frac{1}{4}
	\end{array}
	\right)$. Clearly $e$ is invertible and $A$ is unital. So by Proposition \ref{exis unit}, $A_{e}$ is unital. It is well-known that $A$ is biprojective, see \cite{run}. So $A$ is approximately biprojective. Applying previous theorem $A_{e}$ becomes approximately biprojective.
\end{Example}
\begin{Definition}
We say that a Banach algebra $A$ has approximate (F)-property(or $A$ is AFP) if  there is an approximate $A-$bimodule morphsim $(\rho_{\alpha})$ from $A$ into $(A\otimes_{p}A)^{**}$ such that $\pi^{**}_{A}\circ\rho_{\alpha}(a)-a\rightarrow 0,$ for each $a\in A.$
\end{Definition}
 For the motivation of this definition see \cite{askar}.
 \begin{Proposition}
 	If $A$ is AFP and $A_{e}$ is unital, then $A_{e}$ is approximately biprojective.
 \end{Proposition}
\begin{proof}
Since $A$ is AFP, there exists an approximate $A-$bimodule morphsim $(\rho_{\alpha})$ from $A$ into $(A\otimes_{p}A)^{**}$ such that $\pi^{**}_{A}\circ\rho_{\alpha}(a)-a\rightarrow 0,$ for each $a\in A.$ It is easy to see that $(\rho_{\alpha})$  is an approximate $A_{e}-$bimodule morphsim from $A_{e}$ into $(A_{e}\otimes_{p}A_{e})^{**}$ such that $\pi^{**}_{A_{e}}\circ\rho_{\alpha}(a)-a\rightarrow 0,$ for each $a\in A_{e}.$  Let $T:A_{e}\otimes_{p}A_{e}\rightarrow A_{e}\otimes_{p}A_{e}$ be the same map as in the proof of Theorem \ref{Approximate homological}. Clearly $T$ is $A_{e}$-module morphism, so is $T^{**}$. Similar to the proof of Theorem \ref{Approximate homological}, for the net $(T^{**}\circ \rho_{\alpha})$ is an approximate $A_{e}-$bimodule morphism from $A_{e}$ into $(A_{e}\otimes_{p}A_{e})^{**}$ such that
$$\pi^{**}_{A_{e}}\circ T^{**}\circ\rho_{\alpha}(a)-a=\pi^{**}_{A}\circ \rho_{\alpha}(a)-a\rightarrow 0,\quad (a\in A).$$ We denote the identity of $A_{e}$
with $a_{0}$ and define $m_{\alpha}=\rho_{\alpha}(a_{0})$. Clearly $(m_{\alpha})$ is a net in $(A_{e}\otimes_{p}A_{e})^{**}$ which satisfies $$a\ast m_{\alpha}-m_{\alpha}\ast a\rightarrow 0,\qquad \pi_{A_{e}}^{**}(m_{\alpha})a-a\rightarrow 0,\quad (a\in A_{e}).$$
 Take $\epsilon>0$ and arbitrary finite subsets $F\subseteq A_{e}$,
$\Lambda\subseteq (A_{e}\otimes_{p}A_{e})^{*}$ and $\Gamma \subseteq A_{e}^*$.
Then we have
$$||a\ast
m_{\alpha}-m_{\alpha}\ast a||<\epsilon,\quad
||\pi_{A_{e}}^{**}(m_{\alpha})a-a||<\epsilon,\quad (a\in
F).$$ It is well-known that for each $\alpha$, there exists a net
$(n^{\alpha}_{\beta})_{\beta}$ in $A_{e}\otimes_{p}A_{e}$ such that
$n^{\alpha}_{\beta}\xrightarrow{w^{*}}m_{\alpha}$. Since
$\pi^{**}_{A_{e}}$ is a $w^{*}$-continuous map, we have
$$\pi_{A_{e}}(n^{\alpha}_{\beta})=\pi_{A_{e}}^{**}(n^{\alpha}_{\beta})\xrightarrow{w^{*}}\pi_{A}^{**}(m_{\alpha}).$$
Thus we have $$|a\ast
n^{\alpha}_{\beta}(f)-a\ast m_{\alpha}(f)|<\frac{\epsilon}{K_{0}},\quad
|
n^{\alpha}_{\beta}\ast a(f)-m_{\alpha}\ast(f)|<\frac{\epsilon}{K_{0}}$$
and
$$|\pi_{A_{e}}(n^{\alpha}_{\beta})(g)-\pi_{A_{e}}^{**}(m_{\alpha})(g)|<\frac{\epsilon}{K_{1}},$$
for each $ a\in F$, $ f\in\Lambda$ and $g\in A^*$, where $K_{0}=\sup\{||f||:f\in
\Lambda\}$ and $K_{1}=\sup\{||g||:g\in
\Gamma\}$. Since $a\ast m_{\alpha}-m_{\alpha}\ast a\rightarrow 0$ and  $\pi_{A_{e}}^{**}(m_{\alpha})\ast a-a\rightarrow 0,$ we can find
$\beta=\beta(F,\Lambda,\Gamma,\epsilon)$ such that
$$|a\ast n^{\alpha}_{\beta}(f)-n^{\alpha}_{\beta}\ast a(f)|<c\frac{\epsilon}{K_{0}},\quad |\pi_{A_{e}}(n^{\alpha}_{\beta})\ast a(g)-a(g)|<\frac{\epsilon}{K_{1}},\quad (a\in F,f\in \Lambda,g\in \Gamma)$$
for some $c\in \mathbb{R^{+}}$. Using Mazur's lemma, we have a net
$(n_{(F,\Lambda,\Gamma,\epsilon)})$ in $A_{e}\otimes_{p}A_{e}$ such that $$||a\ast
n_{(F,\Lambda,\Gamma,\epsilon)}-n_{(F,\Lambda,\Gamma,\epsilon)}\ast a||\rightarrow
0,\quad || \pi_{A}(n_{(F,\Lambda,\Gamma,\epsilon)})\ast a-a||\rightarrow
0,\quad (a\in F).$$ Define $\rho_{(F,\Lambda,\Gamma,\epsilon)}:A_{e}\rightarrow
A_{e}\otimes_{p}A_{e}$ by $\rho_{(F,\Lambda,\Gamma,\epsilon)}(a)=a\ast
n_{(F,\Lambda,\Gamma,\epsilon)} $ for each $a\in A_{e}.$ It is clear that
$\rho_{(F,\Lambda,\Gamma,\epsilon)}(a\ast b)=a\ast
\rho_{(F,\Lambda,\Gamma,\epsilon)}(b)$ for each $a,b\in A.$ Also
\begin{equation}
\begin{split}
||\rho_{(F,\Lambda,\epsilon)}(a\ast b)-\rho_{(F,\Lambda,\Gamma,\epsilon)}(a)\ast b||&=||ab\ast
n_{(F,\Lambda,\Gamma,\epsilon)}-a\ast( n_{(F,\Lambda,\Gamma,\epsilon)}\ast b)||\\
&\leq ||a||||b\ast n_{(F,\Lambda,\Gamma,\epsilon)}-
n_{(F,\Lambda,\Gamma,\epsilon)}\ast b||\rightarrow 0,
\end{split}
\end{equation}
for each $a,b\in A_{e}.$ Also
\begin{equation}
\begin{split}
||\pi_{A_{e}}\circ\rho_{(F,\Lambda,\Gamma,\epsilon)}(a)-a||&=||\pi_{A_{e}}(a\ast
n_{(F,\Lambda,\Gamma,\epsilon)})-a||\\&=||\pi_{A_{e}}(a\ast
n_{(F,\Lambda,\Gamma,\epsilon)})-\pi_{A_{e}}(
n_{(F,\Lambda,\Gamma,\epsilon)}\ast a)+\pi_{A_{e}}(
n_{(F,\Lambda,\Gamma,\epsilon)}\ast a)-a||\\
&\leq ||\pi_{A_{e}}(a\ast
n_{(F,\Lambda,\Gamma,\epsilon)})-\pi_{A_{e}}(
n_{(F,\Lambda,\Gamma,\epsilon)}\ast a)||+||\pi_{A_{e}}(
n_{(F,\Lambda,\Gamma,\epsilon)})\ast a-a||\\
&\rightarrow 0,
\end{split}
\end{equation}
for each $a\in F$. Thus with respect to the net $(\rho_{(F,\Lambda,\Gamma,\epsilon)})_{(F,\Lambda,\Gamma,\epsilon)},$ $A_{e}$ becomes approximately biprojective.

	\end{proof}
A Banach algebra $A$ is called Johnson pseudo-contractible,  if
there exists a not necessarily bounded net $(m_{\alpha})$ in
$(A\otimes_{p}A)^{**}$ such that $a\cdot m_{\alpha}=m_{\alpha}\cdot
a$ and $\pi^{**}_{A}(m_{\alpha})a-a\rightarrow 0,$ for every $a\in
A,$ see \cite{sah new amen1} and \cite{sah1}. 

A Banach algebra $A$ is called biflat,   if  there is a bounded $A-$bimodule morphsim $\rho$ from $A$ into $(A\otimes_{p}A)^{**}$ such that $\pi^{**}_{A}\circ\rho_{\alpha}(a)=a$, for each $a\in A$, see \cite{run}.

\begin{Proposition}
Let $A$ be a Banach algebra and $e\in \overline{B^{0}_{1}}$. Suppose that $A_{e}$ is unital.
Then  $A$ is Johnson pseudo-contractible if and only if  $A_{e}$ is Johnson pseudo-contractible.
\end{Proposition}
\begin{proof}
Since $A_{e} $ is unital, by Proposition \ref{exis unit} $A$ is unital.  So using \cite[Theorem 2.1]{askar}, Johnson pseudo-contractibility of $A$ implies that $A$ is amenable. Thus by \cite[Exercise 4.3.15]{run}, $A$ is biflat. Then by \cite[Theorem 2.4]{khod} $A_{e}$ is biflat. Since $A_{e}$ is unital, biflatness of $A_{e}$	 gives the amenability of $A_{e}$.

For converse, suppose that $A_{e}$ is Johnson pseudo-contractible. Since $A_{e}$ is unital by \cite[Theorem 2.1]{askar} $A_{e}$ is amenable, so is biflat. Applying \cite[Theorem 2.4]{khod}  follows that $A$ is biflat. Using Proposition \ref{exis unit}, $A$ is unital, thus by \cite[Exercise 4.3.15]{run} $A$ is amenable. So \cite[Lemma 2.1]{sah new amen1} implies that $A$ is Johnson pseudo-contractible.
	
	\end{proof}
\begin{Example}
Let  $A=\{\left(\begin{array}{ccc} a_{11}&a_{12}&a_{13}\\
	0&a_{22}&a_{23}\\
	0&0&a_{33}
	\end{array}
	\right)| a_{ij}\in \mathbb{C}\}$ and suppose that $e=\left(\begin{array}{ccc} \frac{1}{4}&\frac{1}{4}&0\\
	0&\frac{1}{4}&0\\
	0&0&\frac{1}{4}
	\end{array}
	\right)$. Clearly $e$ is invertible and $A$ is unital. So by Proposition \ref{exis unit} $A_{e}$ is unital. Using \cite[Theorem 2.5]{sah new amen1} we know that $A$ is not Johnson pseudo-contractible. So by previous proposition
	$A_{e}$ is not Johnson pseudo-contractible. 
	\end{Example}
%------------------------------------------------------------------------------------------------------------------------------------------
\begin{small}

\end{small}
\end{document}